\documentclass[12pt]{amsart}
\usepackage{fourier}
\usepackage[T1]{fontenc}

\usepackage{graphics}
\usepackage{amsfonts,amssymb,color}
\usepackage[mathscr]{eucal}
\usepackage{amsmath, amsthm}
\usepackage{mathrsfs}
\usepackage{amsbsy}
\usepackage{dsfont}
\usepackage{bbm}
\usepackage{wasysym}
\usepackage{stmaryrd}
\usepackage{url}

\input xypic
\xyoption{all}

\makeindex
\makeglossary

\begin{document}
\baselineskip = 16pt

\newcommand \ZZ {{\mathbb Z}}
\newcommand \NN {{\mathbb N}}
\newcommand \RR {{\mathbb R}}
\newcommand \PR {{\mathbb P}}
\newcommand \AF {{\mathbb A}}
\newcommand \GG {{\mathbb G}}
\newcommand \QQ {{\mathbb Q}}
\newcommand \CC{{\mathbb C}}
\newcommand \bcA {{\mathscr A}}
\newcommand \bcC {{\mathscr C}}
\newcommand \bcD {{\mathscr D}}
\newcommand \bcF {{\mathscr F}}
\newcommand \bcG {{\mathscr G}}
\newcommand \bcH {{\mathscr H}}
\newcommand \bcM {{\mathscr M}}
\newcommand \bcJ {{\mathscr J}}
\newcommand \bcL {{\mathscr L}}
\newcommand \bcO {{\mathscr O}}
\newcommand \bcP {{\mathscr P}}
\newcommand \bcQ {{\mathscr Q}}
\newcommand \bcR {{\mathscr R}}
\newcommand \bcS {{\mathscr S}}
\newcommand \bcV {{\mathscr V}}
\newcommand \bcW {{\mathscr W}}
\newcommand \bcX {{\mathscr X}}
\newcommand \bcY {{\mathscr Y}}
\newcommand \bcZ {{\mathscr Z}}
\newcommand \goa {{\mathfrak a}}
\newcommand \gob {{\mathfrak b}}
\newcommand \goc {{\mathfrak c}}
\newcommand \gom {{\mathfrak m}}
\newcommand \gon {{\mathfrak n}}
\newcommand \gop {{\mathfrak p}}
\newcommand \goq {{\mathfrak q}}
\newcommand \goQ {{\mathfrak Q}}
\newcommand \goP {{\mathfrak P}}
\newcommand \goM {{\mathfrak M}}
\newcommand \goN {{\mathfrak N}}
\newcommand \uno {{\mathbbm 1}}
\newcommand \Le {{\mathbbm L}}
\newcommand \Spec {{\rm {Spec}}}
\newcommand \Gr {{\rm {Gr}}}
\newcommand \Pic {{\rm {Pic}}}
\newcommand \Jac {{{J}}}
\newcommand \Alb {{\rm {Alb}}}
\newcommand \Corr {{Corr}}
\newcommand \Chow {{\mathscr C}}
\newcommand \Sym {{\rm {Sym}}}
\newcommand \alb {{\rm {alb}}}
\newcommand \Prym {{\rm {Prym}}}
\newcommand \cha {{\rm {char}}}
\newcommand \eff {{\rm {eff}}}
\newcommand \tr {{\rm {tr}}}
\newcommand \Tr {{\rm {Tr}}}
\newcommand \pr {{\rm {pr}}}
\newcommand \ev {{\it {ev}}}
\newcommand \cl {{\rm {cl}}}
\newcommand \interior {{\rm {Int}}}
\newcommand \sep {{\rm {sep}}}
\newcommand \td {{\rm {tdeg}}}
\newcommand \alg {{\rm {alg}}}
\newcommand \im {{\rm im}}
\newcommand \gr {{\rm {gr}}}
\newcommand \op {{\rm op}}
\newcommand \Hom {{\rm Hom}}
\newcommand \Hilb {{\rm Hilb}}
\newcommand \Sch {{\mathscr S\! }{\it ch}}
\newcommand \cHilb {{\mathscr H\! }{\it ilb}}
\newcommand \cHom {{\mathscr H\! }{\it om}}
\newcommand \colim {{{\rm colim}\, }} 
\newcommand \End {{\rm {End}}}
\newcommand \coker {{\rm {coker}}}
\newcommand \id {{\rm {id}}}
\newcommand \van {{\rm {van}}}
\newcommand \spc {{\rm {sp}}}
\newcommand \Ob {{\rm Ob}}
\newcommand \Aut {{\rm Aut}}
\newcommand \cor {{\rm {cor}}}
\newcommand \Cor {{\it {Corr}}}
\newcommand \res {{\rm {res}}}
\newcommand \red {{\rm{red}}}
\newcommand \Gal {{\rm {Gal}}}
\newcommand \PGL {{\rm {PGL}}}
\newcommand \Bl {{\rm {Bl}}}
\newcommand \Sing {{\rm {Sing}}}
\newcommand \spn {{\rm {span}}}
\newcommand \Nm {{\rm {Nm}}}
\newcommand \inv {{\rm {inv}}}
\newcommand \codim {{\rm {codim}}}
\newcommand \Div{{\rm{Div}}}
\newcommand \sg {{\Sigma }}
\newcommand \DM {{\sf DM}}
\newcommand \Gm {{{\mathbb G}_{\rm m}}}
\newcommand \tame {\rm {tame }}
\newcommand \znak {{\natural }}
\newcommand \lra {\longrightarrow}
\newcommand \hra {\hookrightarrow}
\newcommand \rra {\rightrightarrows}
\newcommand \ord {{\rm {ord}}}
\newcommand \Rat {{\mathscr Rat}}
\newcommand \rd {{\rm {red}}}
\newcommand \bSpec {{\bf {Spec}}}
\newcommand \Proj {{\rm {Proj}}}
\newcommand \pdiv {{\rm {div}}}
\newcommand \CH {{\it {CH}}}
\newcommand \wt {\widetilde }
\newcommand \ac {\acute }
\newcommand \ch {\check }
\newcommand \ol {\overline }
\newcommand \Th {\Theta}
\newcommand \cAb {{\mathscr A\! }{\it b}}

\newenvironment{pf}{\par\noindent{\em Proof}.}{\hfill\framebox(6,6)
\par\medskip}

\newtheorem{theorem}[subsection]{Theorem}
\newtheorem{conjecture}[subsection]{Conjecture}
\newtheorem{proposition}[subsection]{Proposition}
\newtheorem{lemma}[subsection]{Lemma}
\newtheorem{remark}[subsection]{Remark}
\newtheorem{remarks}[subsection]{Remarks}
\newtheorem{definition}[subsection]{Definition}
\newtheorem{corollary}[subsection]{Corollary}
\newtheorem{example}[subsection]{Example}
\newtheorem{examples}[subsection]{examples}

\title{Involutions on algebraic surfaces and zero cycles }
\author{Kalyan Banerjee}

\address{Indian Statistical Institute, Bangalore Center, Bangalore 560059}

\email{kalyanb$_{-}$vs@isibang.ac.in}

\begin{abstract}
In this note we are going to consider a smooth projective surface equipped with an involution and study the action of the involution at the level of Chow group of zero cycles.
\end{abstract}

\maketitle

\section{Introduction}
In this note we want to consider the generalised Bloch conjecture \cite{Vo}  [conjecture 11.19], which says that the action of a degree two correspondence on the Chow group of zero cycles on a smooth projective surface is determined by its cohomology class in $H^4(S\times S,\ZZ)$. This is equivalent to the following: let $\Gamma$ be a correspondence of codimension $2$ on $S\times T$ where $S,T$ are smooth projective surfaces over the field of complex numbers. Suppose that $\Gamma^*$ vanishes on $H^0(T,\Omega^2_T)$ then the homomorphism $\Gamma_*$ from $\CH_0(S)$ to $\CH_0(T)$ vanishes on the kernel of the albanese map $alb_S:\CH_0(S)\to Alb(S)$.

In  \cite{Voi}, the conjecture was proved for a symplectic involution on a $K3$ surface. In this paper the author consider an automorphism of order two $i$ of the given K3 surface, such that $i^*$ acts as identity on globally holomorphic $2$-forms, then $i_*$ acts as identity on $\CH_0$ of the K3 surface. Also the similar question was considered in \cite{G} for intersection of quadrics and cubics in $\PR^4$ which are examples of K3 surfaces. Also in
\cite{HK} the question was considered and proved for certain examples of K3 surfaces equipped with a symplectomorphism.

In this note we prove the following theorem:
\smallskip

\begin{theorem}
Let $S$ be a smooth surface with $p_g=q=0$ and  having an involution $i$ on it such that $S/i$  has an elliptic pencil on it. Then the involution $i$ acts as $-1$ on the group of degree zero zero cycles on $S$, modulo rational equivalence (denoted by $A_0(S)$).
\end{theorem}

\smallskip

Our method is the same as in the proof of Bloch's conjecture for surfaces not of general type with $p_g=0$ as in \cite{BKL}. 

In the next section we try to understand the action of the involution on the Chow group of a general smooth, projective surface. This we do by reducing the problem to a very general hyperplane section of the surface itself. Precisely speaking let $\tau$ be the involution on the surface $S$. Let $j_t$ denote the embedding of a very general hyperplane section $C_t$ into $S$ and $j_{\tau (t)}$ the embedding of $\tau(C_t)$ into $S$. Then we consider  the kernel of the bi-Gysin homomorphism $j_{t*}+j_{\tau(t)*}$ from $J(C_t)\times J(\tau(C_t))$ to $A^2(S)$ ($A^i$ is the group of algebraically trivial codimension $i$ algebraic cycles modulo rational equivalence, it is also denoted by $A_{n-i}$, where $n$ is the dimension of the ambient variety). By monodromy argument we prove that the kernel is either countable or $\tau$ acts as $-1$ on the image of $J(C_t)$ under $j_{t*}$ into $A^2(S)$. Similar techniques will show that the kernel of $j_{t*}-j_{\tau(t)*}$ is either countable or $\tau$ acts as identity on the image of $J(C_t)$. Then combining these two facts we show that for a surface with $p_g>0$ and $q=0$, we have three possibilities:

\textit{ Either both the kernels $j_{t*}+j_{\tau(t)*},j_{t*}-j_{\tau(t)*}$ are countable or $\tau$ acts as $-id$ on the image of $J(C_t)$ under $j_{t*}$ or as identity on the image of $J(C_t)$ under $j_{t*}$.}

\smallskip 

So if we can exclude the first possibility then we know the action of $\tau$ on $A^2(S)$, by studying the action of $\tau$ on the Jacobian of a very general hyperplane section.

{\small \textbf{Acknowledgements:} The author would like to thank the ISF-UGC project for funding this project and also thanks the hospitality of Indian Statistical Institute, Bangalore Center for hosting this project. The author is indebted to Ramesh Sreekantan for suggesting this problem to the author and for many helpful discussions on the theme of the paper. Lastly the author is grateful to Chuck Weibel for constructive criticism on improving the exposition of the paper and for his advice to improve \ref{theorem2} and to generalize theorem \ref{theorem 1}.

We assume that the ground field is algebraically closed and of characteristic zero.}


\section{The Bloch-Kas-Liebarman technique}

In this section we prove the following theorem:

\begin{theorem}
\label{theorem2}
Let $S$ be a smooth surface  admitting a $2:1$ map $f$ to a surface $F$ with $p_g=q=0$ and admitting an elliptic pencil. Let $i$ be the involution on $S$ arising from the $2:1$ map. Then the group of invariants of $A_0(S)$, given by
$$\{z\in A_0(S):i(z)=z\}$$
is finite dimensional.
\end{theorem}
\begin{proof}
To prove that the group of $i$-invariants of  the Chow group of degree zero cycles of $S$, we follow the Bloch-Kas-Lieberman technique as presented in \cite{BKL}. First consider the pencil of elliptic curves on the surface $F$. That is a  map from $F\dashrightarrow L$, where $L$ is isomorphic to $\PR^1$.

Suppose that a pencil of curves on a surface $F$ can be given by choosing a projective line $L$ in $\PR(H^0(F,D)):=|D|$, where $D$ is a line bundle on $F$. So every element $t$ of this projective line gives rise to a global section $\sigma_t$ of $D$, which is non-zero and well-defined upto scalar multiplication. Let $F_t$ be the curve in $F$ defined by the zero locus of $\sigma_t$. Now let $\sigma_0,\sigma_{\infty}$ be two linearly independent global sections spanning the two dimensional vector subspace of $H^0(F,D)$, underlying the line $L$. Then any element $\sigma_t$ in this vector space\ look like $\sigma_0+t\sigma_{\infty}$. Now the rational map $F\to L$ is defined by
$$x\mapsto [\sigma_0(x):\sigma_{\infty}(x)]$$
and it is not defined along the common zero locus of $\sigma_0=\sigma_{\infty}=0$.
Consider the surface
$$\wt{F}=\{(x,t)\in F\times L|x\in F_t\}$$
this is nothing but the blow-up of $F$ along the base locus of the above rational map. Then sending $(x,t)$ to $t$ defines a regular map from $\wt{F}$ to $L$. That is we blow up the base locus of the rational map $F\dashrightarrow L$. Now consider the pull-back of $\wt{F}\to F$ to $S$, call it $\wt{S}$. Then $\wt{S}$ is nothing but the blow up of $S$ along the base locus of the rational map $S\dashrightarrow L$. Observe that fixing a point $0$ in on $F$, which is in the base locus of the pencil, we have  a section of $\wt{F}\to L$ given by $t\mapsto (0,t)$. Let us continue to denote the map from $\wt{S}$ to $\wt{F}$ by $f$.

Consider the Jacobian fibration $J\to L$ corresponding to $\wt{F}\to L$. Now fix a smooth hyperplane section $Y$ of $\wt{F}$ under the embedding of $\wt{F}$ in some $\PR^N$. Let $\pi$ be the morphism from $\wt{S}\to L$ and $\pi'$ is from $\wt{F}\to L$. Let us have
$$Y\cap \pi'^{-1}(t)=\sum_{i=1}^n p_i(t)$$
Then we have a  map $g$ from $\wt{S}$ to $J$
$$q\mapsto \alb_{t}(nf(q)-\sum_i p_i(\pi(q)))$$
where $\alb_t$ is the Albanese map from $F_t$ to $J_t$, $t=\pi(q)$. It is defined because the pencil $\wt{F}\to L$ has a section. This  map is dominant as it is dominant on fibers.

\begin{lemma}
Let $T(J)$ denote the albanese kernel for $J$. If $T(J)=0$, then the group of invariants under the action of $i$, in $A_0(S)$ is finite dimensional.
\end{lemma}
\begin{proof}
The proof of this lemma follows by arguing as in \cite{BKL}[proposition 4]. To prove the claim we have to  understand the quasi-inverse of  $g$ given by a correspondence on $\wt{S}\times J$. Let $\alpha$ belong to $J$ that lies over $t\in L$. View $\alpha$ as a zero cycle on $J_t$, that is it is an element in $\Pic^0(J_t)$ (by using the section for the Jacobian fibration). Since $\Pic^0(J_t)$ is isomorphic to $J_t$, there is a unique point in $q_{i}(t)$ on $F_t$ such that $q_i(t)-p_i(t)$ is rationally equivalent to  $\alpha$. Now $f^{-1}(q_i(t))=\{q_i'(t),q_i''(t)\}, f^{-1}(p_i(t))=\{p_i'(t),p_i''(t)\}$. So we can define $\lambda$ to be
$$\alpha\mapsto \sum_i (q_i'(t)+q_i''(t))-(p_i'(t)+p_i''(t)).$$

Let $q-p$ be a zero cycle where $q,p$ are closed points on $\wt{S}$.  Then we compute $\lambda g_*(q-p+i(q-p))$. We have by definition
$$g_*(q+iq-p-ip)=nf_{*}(q+iq-p-ip)-\sum_{i=1}^n 2 (p_i(\pi(q))-p_i(\pi(p))\;,$$
let $f_(q)=q', f(p)=p'$.
Then
$$\lambda g_*(q+iq-p-i(p))=\lambda[(2nq'-2np')-(2\sum_{i=1}^n (p_{i}(\pi(q)))-p_i(\pi(p)))]$$
which can be re-written as
$$2\sum_{i=1}^n \lambda(q'-(p_{i}(\pi(q))))+\lambda(p'-(p_{i}(\pi(p))))\;.$$
Now $\lambda(q'-(p_{i}(\pi(q))))=(q+iq-p_i'(\pi(q))-p_i''(\pi(q)))$.
Therefore $\lambda g_*(q+iq-p-i(p))$ is equal to
$$2n(q+i(q)-p-i(p))-a$$
where $a$ is a zero cycle supported on $Y'=f^{-1}(Y)$. So for general hyperplane section $Y$ of $\wt{F}$, we have $Y'$ a smooth projective curve.
Therefore by Chow moving lemma, for any zero cycle $z$ of degree zero on $\wt{S}$, we have
$$2n(z+iz)-\lambda g_*(z)\;,$$
is supported on the Jacobian of $Y'$.
Suppose $z$ belongs to the albanese kernel $T(\wt{S})$. Then we have $g_*(z)$ belonging to the albanese kernel  $T(J)$. If $T(J)$ is zero then we can conclude that $g_*(z)$ is rationally equivalent to $0$ on $J$. So we have $g_*(i(z))$ is also rationally equivalent to zero. Composing with $\lambda$ we get that $2n(z+iz)$ is supported on $J(Y')$. Tensoring with $\QQ$ we get that
$z+iz$ is supported on $J(Y')\otimes _{\ZZ}\QQ$. So it means that the group $T(\wt{S})^i$ of $i$-invariant elements in $T(\wt{S})$, is finite dimensional (rationally): in  the sense that there exists a smooth projective curve $C$, and a correspondence $\Gamma$ on $C\times \wt{S}$ such that $\Gamma_*$ from $J(C)\otimes_{\ZZ}\QQ$ to $T(\wt{S})^i\otimes _{\ZZ}\QQ$ is surjective. Since the albanese map is surjective it follows that the group of $i$-invariant elements of $A_0(\wt{S})\otimes_{\ZZ}\QQ$ is finite dimensional.  Then by lemma 3.1 in \cite{GG} it follows that the group of $i$-invariant elements of $A_0(\wt{S})$ is finite dimensional. Since $\wt{S}$ is a blow-up of $S$ and finite dimensionality is a birational invariant, we have the group of $i$-invariant elements of $A_0(S)$ is finite dimensional.
\end{proof}

Now we repeat the proof  that $T(J)=0$ by showing that $J$ is rational following \cite{BKL}.

\begin{lemma}
$$T(J)=0$$
\end{lemma}

\begin{proof}
For a proof please see \cite{BKL}[Proposition 4] or \cite{SS}. The proof basically relies on the fact that an elliptic surface with $p_g=0$ is rational.
\end{proof}
\end{proof}

\begin{corollary}
\label{cor1}
Let $S$ be a surface of general type of geometric genus zero and with an involution $i$, such that $S/i$ is smooth with $p_g=0=q$ and has an elliptic pencil on it. Then the  group $A_0(S/i)$ is finite dimensional.
\end{corollary}
\begin{proof}
The surface $S$ satisfies the condition of the previous theorem \ref{theorem2}, hence the conclusion.
\end{proof}

\begin{corollary}
Let $S$ be a surface with $p_g=q=0$. Let $i$ be an involution on $S$ such that $S/i$ is smooth. Then the involution acts as $-1$ on $A_0(S)$.
\end{corollary}

\begin{proof}
Since $A_0(S/i)$ is finite dimensional, it follows by \cite{Voi}[Theorem 2.3] that the homomorphism $\id+i_*$ factors through Albanese of $S$, which is zero. Hence the involution acts as $-1$ on $A_0(S)$. 
\end{proof}

\begin{remark}
The above corollary \ref{cor1} is useful to prove the Bloch's conjecture on $S/i$, when it has an elliptic pencil.
\end{remark}

\section{Curves on a surface and monodromy}
Let $S$ be a smooth, projective, surface over $\CC$. Let us fix an embedding of $S$ inside $\PR^N$. Let $\tau$ be an involution acting on $S$. Let $t$ be a closed point in ${\PR^N}^*$. Consider the corresponding hyperplane $H_t$ inside $\PR^N$ and consider its intersection with $S$. Then we get a curve $C_t$ inside $S$ and $\tau(C_t)$ inside $S$. By Bertini's theorem, a general such hyperplane section of $S$ will be smooth and irreducible. Now consider two  curves $C_t,\tau(C_t)$ in $S$. Let $g$ be the genus of $C_t$. Then we have the following commutative diagram.

$$
  \diagram
   \Sym^g C_t\times \Sym^g (\tau(C_t))\ar[dd]_-{} \ar[rr]^-{} & & \Sym^{2g}S\ar[dd]^-{} \\ \\
  A_0(C_t)\times A_0(\tau(C_t)) \ar[rr]^-{} & & A_0(S)
  \enddiagram
  $$
Here the morphism from $\Sym^g C_t\times \Sym^g \tau(C_t)$ to $\Sym^{2g}S$ is given by
$$(\sum_i P_i,\sum_j Q_j)=\sum_i P_i+\sum_j Q_j$$
and the homomorphism from $A_0(C_t)\times A_0(\tau(C_t))$ to $A_0(S)$
is given by
$$j_{t,\tau(t)*}=j_{t*}+j_{\tau(t)*}\;.$$
It is easy to see that the above diagram is commutative (since $\CC$ is algebraically closed). By the Abel-Jacobi theorem $A_0(C_t)\times A_0(\tau(C_t))$ is isomorphic to $J(C_t)\times J(\tau(C_t))$. Following the argument of \cite{BG}, proposition $6$ we get that the kernel of $j_{t,\tau(t)*}$ is a countable union of translates of an abelian subvariety of $J(C_t)\times J(\tau(C_t))$. Call this abelian subvariety $A_{t}$. Assume that $H^3(S,\QQ)=0$. Now we prove the following:
\begin{theorem}
\label{theorem 1}
For a very general $t$, the abelian variety $A_{t}$ is either $\{0\}$ or $J(C_t)\times\{0\}$ or $\{0\}\times J(\tau(C_t))$ or the diagonal $J(C_t)$,(induced by the diagonal embedding $C\to C\times \tau(C)$) or $J(C_t)\times J(\tau(C_t))$.
\end{theorem}
\begin{proof}
The argument comes from monodromy. Let $L$ denote a Lefschetz pencil on $S$, that is a line in ${\PR^N}^{*}$, such that the corresponding hyperplane sections with $S$ are either smooth or has ordinary double point singularities. Let $0_1,\cdots,0_m$ are the points on $L$ such that the corresponding fibers on $S$ are singular. We have a natural monodromy representation of the fundamental group of $L\setminus \{0_1,\cdots,0_m\}$ on the Gysin kernel at the level of cohomology, which is $H^1(C_t,\QQ)$ and hence on $H^1(\tau(C_t),\QQ)\cong H^1(C_t,\QQ)$ respectively, for a very general $t$ such that $C_t,\tau(C_t)$ are smooth. By theorem 3.27 in \cite{Vo} we have that these monodromy representations are irreducible. So it will follow that the induced representation of $G=\pi_1(L\setminus \{0_1,\cdots,0_m\},t)$ on $H^1(C_t,\QQ)\oplus H^1(\tau(C_t),\QQ)$ has the following property. Any $G$ invariant subspace of it is either $\{0\}$ or $H^1(C_t,\QQ)$ or $H^1(\tau(C_t),\QQ)$ or all of $H^1(C_t,\QQ)\oplus H^1(\tau(C_t),\QQ)$. Consequently, by using the correspondence between polarized Hodge structures of weight one and abelian varieties we have that the only non-trivial  abelian subvarieties of $J(C_t)\times J(\tau(C_t))$ are either $J(C_t)\times\{0\}$ or $\{0\}\times J(\tau(C_t))$ or the diagonal $J(C_t)$ or all of $J(C_t)\times J(\tau(C_t))$. Now to prove that $A_{t}$ is either one of these four possibilities or it is trivial, we have to show that the Hodge structure corresponding to $A_{t}$ is $G$ equivariant.  So for a very general $t$ we have an abelian subvariety $A_{t}$ of $J(C_t)\times J(\tau(C_t))$. Now consider the morphism of $\CC$ embedded in $\CC(t)$ and view $A_{t}$ and $J(C_t)\times J(\tau(C_t))$ as abelian varieties over ${\CC(t)}$. Let $K$ be the minimal field of definition of $A_{t}$ and $J(C_t)\times J(\tau(C_t))$ in $\overline{\CC(t)}$. Since $K$ is finitely generated over $\CC(t)$ and contained in $\overline{\CC(t)}$ we have $K$ a finite extension of $\CC(t)$. Let $C'$ be a curve  such that $\CC(C')$ is isomorphic to $K$ and $C'$ maps finitely onto $L$. Then we have $A_{t}$ and $J(C_t)\times J(\tau(C_t))$ defined over $K$ and we can spread $A_{t}$ and $J(C_t)\times J(\tau(C_t))$ over some Zariski open $U$ in $C'$. Call these spreads as $\bcA,\bcJ$. Then throwing out some more points from $U$ we will get that the morphism from $\bcA,\bcJ$ to $U$ are proper, submersions of smooth manifolds, if we view everything over $\CC$. Then by Ehressmann's theorem we have two fibrations $\bcA\to U$ and $\bcJ\to U$. Since any fibration gives rise to a local system and hence a monodromy representation of the fundamental group of $\pi_1(U,t')$ on $H^{2d-1}(A_{t},\QQ),H^{4g-1}(J(C_t)\times J(\tau(C_t)),\QQ)\cong H^1(C_t,\QQ)\oplus H^1(\tau(C_t),\QQ)$ where $d,2g$ are dimensions of $A_{t},J(C_t)\times J(\tau(C_t))$. Now $\pi_1(U,t')$ is a finite index subgroup of $G$. We prove that $H=H^{2d-1}(A_{t},\QQ)$ is a $G$-equivariant subspace of $H^1(C_t,\QQ)\oplus H^1(\tau(C_t),\QQ)$ (since $A_{t}$ is a sub-abelian variety of $J(C_t)\times J(\tau(C_t))$, $H^{2d-1}(A_{t},\QQ)$ is a subspace of $H^1(C_t,\QQ)\oplus H^1(\tau(C_t),\QQ)$). Now $G$ acts on $H^1(C_t,\QQ)\oplus H^1(\tau(C_t),\QQ)$ by the Picard-Lefschtez formula, that is
$$\gamma. (\alpha+\beta)=\gamma.\alpha+\gamma.\beta\;.$$
By definition the above is equal to
$$\alpha-\langle\alpha,\delta_{\gamma}\rangle\delta_{\gamma}+\beta
-\langle \beta,\delta_{\gamma}\rangle\delta_{\gamma}\;. $$
Now suppose that $\alpha+\beta$ belongs to $H$. We have to prove that for all $\gamma$ in $G$, $\gamma.(\alpha+\beta)$ belongs to $H$.
Consider $$(\gamma)^m(\alpha+\beta)=m\alpha-m
\langle\alpha,\delta_{\gamma}\rangle\delta_{\gamma}+
m\beta-m\langle\beta,\delta_{\gamma}\rangle\delta_{\gamma}
$$
$\delta_{\gamma}$ is the vanishing cycles corresponding to $\gamma$.
Since $\gamma^m$ is in $\pi_1(U,t')$ we have
$$\gamma^m(\alpha+\beta)-m\alpha-m\beta$$
is in $H$. That would mean that
$$m\langle\alpha,\delta_{\gamma}\rangle\delta_{\gamma}
+m\langle\beta,\delta_{\gamma}\rangle \delta_{\gamma}$$
is in $H$, by applying the Picard Lefschetz once again we get that
$$\gamma.(\alpha+\beta)$$
is in $H$. So $H$ is $G$ equivariant, hence it is either $\{0\}$ or $H^1(C_t,\QQ),H^1(\tau(C_t),\QQ)$ or all of $H^1(C_t,\QQ)\oplus H^1(\tau(C_t),\QQ)$. So the corresponding $A_{t}$ will either be zero or $J(C_t)\times\{0\}$ or $\{0\}\times J(\tau(C_t))$ or $J(C_t)$ or $J(C_t)\times J(\tau(C_t))$.
\end{proof}
This proves that if for one very general $t$, $A_{t}$ is one of the above mentioned possibilities then for another very general $t'$, $A_{t'}$ will achieve the same possibility because after all everything happens in a family.

\begin{remark}

Chuck Weibel has communicated to the author the following generalisation of the above theorem:

\smallskip

We can consider two Lefschetz pencils on $S$. Let $C,C'$ denote the very general fiber of these two pencils respectively. Then the bi-Gysin map $j_{C*}+j_{C'*}$ from $J(C)\times J(C')$ to $A_0(S)$ has the kernel equal to the countable union of translates of an abelian variety $A$ inside the product of the Jacobians. By monodromy this $A$ is either trivial or $J(C)\times \{0\}$ or $\{0\}\times J(C')$ or $J(C)\times J(C')$.
\end{remark}

\subsection{Countability of the Bi-Gysin kernel}
\begin{theorem}
\label{theorem3}
Let $S$ be a smooth projective surface embedded in $\PR^N$. Suppose $A_0(S)$ is not isomorphic to the Albanese variety $Alb(S)$.  Consider a Lefschetz pencil on $S$. Then for a very general $t$, $A_{t}$ is actually $\{0\}$ or the diagonal $J(C_t)$.
\end{theorem}
\begin{proof}
Before going into the proof we mention that a detailed proof of this theorem is an easy generalisation of the theorem 19 in \cite{BG} and its elaboration is present in the PhD thesis of the author in \cite{Ba}[theorem 4.7.1]. Now we proceed to an outline of the proof.

Here we want to prove that if $A_0(S)$ is not weakly representable that is it is not isomorphic to the Albanese variety of $S$ then the kernel of the bi-Gysin homomorphism is countable for a very general hyperplane section $C_t$ of $S$ inside $\PR^N$.

For that consider the family $\bcC\cup \tau(\bcC)$, that is given by
$$\{(x,C_t):x\in C_t\}\cup \{(x,\tau(C_t)):x\in \tau(C_t)\}$$. Then $\bcC_t\cup\tau(\bcC)_t$ is nothing but $C_t\cup \tau(C_t)$.
Suppose that $T$ is the set of all $t$ in ${\PR^N}^{*}$ such that $C_t\cup \tau(C_t)$ is smooth. Let us choose a Lefschetz pencil $L$ in ${\PR^N}^{*}$ such that a very general member $C_t$ of this Lefschetz pencil has the property that $A^2(C_t)\oplus A^2(\tau(C_t))$ is weakly representable (this is trivially true since $C_t$ is a curve). Then for the geometric generic fiber $C_{\bar \eta}$ of $L$, we have
$$A^2(C_{\bar\eta})\oplus A^2(\tau(C_{\bar\eta}))$$
is weakly representable. So there exists a curve $\Gamma_{\bar\eta}$ and a correspondence $Z_{\bar\eta}$ on $\Gamma_{\bar\eta}\times(C_{\bar\eta}\cup \tau(C_{\bar\eta}))$
such that
$$Z_{\bar\eta*}:A^1(\Gamma_{\bar\eta})\to A^2(C_{\bar\eta})\oplus A^2(\tau(C_{\bar\eta}))$$
is surjective. Let $\Gamma_{\bar\eta}$, $Z_{\bar\eta}$ are defined over some finite extension $K$ of $\CC(t)$. Let $C'$ be a smooth projective curve mapping finitely onto $L$ and having function field $K$. Then we can spread $Z_{\bar\eta}$ and $\Gamma_{\bar\eta}$ over some Zariski open $U'$ in $C'$ to get $\bcZ,\bcG$, such that we have
$$\bcZ_*:A^1(\bcG)\to A^2(\bcC_{U'}\cup \tau(\bcC)_{U'})\;.$$

Compactifying and resolving singularities we get $\bcZ'_*$ and $\bcG'$ such that
$$\bcZ'_*:A^1(\bcG')\to A^2(\bcC_{C'}\cup \tau(\bcC)_{C'})\;.$$
We have the following commutative diagram.
$$
  \diagram
  A^1(\bcG')\ar[dd]_-{} \ar[rr]^-{} & & A^2(\bcC_{C'}\cup\tau(\bcC)_{C'})\ar[dd]^-{} \\ \\
  A^1(\Gamma_{\bar\eta'}) \ar[rr]^-{} & & A^2(\bcC_{\eta'}\cup\tau(\bcC_{\eta'}))
  \enddiagram
  $$

By a diagram chase \cite{Ba}[page 121-124] we can prove that for any $\alpha$ in $A^2(\bcC_{C'}\cup \tau(\bcC)_{C'})$ there exists $n_{\alpha}$ an integer such that $n_{\alpha}\alpha$ belongs to the subgroup generated by the image of $\bcZ'_*$ and the kernel of the pull-back $A^2(\bcC_{C'}\cup \tau(\bcC)_{C'})$ to $A^2(\bcC_{\bar\eta}\cup \tau(\bcC_{\bar\eta}))$. This is by theorem 4.7.1 (page 126) in \cite{Ba}.
Now the kernel of the pull-back $A^2(\bcC_{C'}\cup \tau(\bcC)_{C'})\to A^2(\bcC_{\bar \eta}\cup \tau(\bcC)_{\bar\eta})$ tensored with $\QQ$ is
the same as the kernel $A^2(\bcC_{C'}\cup \tau(\bcC)_{C'})\otimes \QQ\to A^2(\bcC_{ \eta'}\cup \tau(\bcC)_{\eta'})\otimes \QQ$. The later group in the above is the colimit of $A^2(\bcC_{W'}\cup \tau(\bcC)_{W'})\otimes\QQ$, $W'$ Zariski open in $C'$. So by the localization exact sequence the kernel is the direct sum of images of the bi-Gysin homomorphisms
$$(j_{t*}+j_{\tau(t)*}):\oplus_{t'\in C'}( A^1(C_t)\oplus A^1(\tau(C_t)))\otimes \QQ\to A^2(\bcC_{C'}\cup \tau(\bcC)_{C'})\otimes \QQ\;.$$

Arguing as in \cite{Ba}[Theorem 4.7.1] we can prove that if the bi-Gysin homomorphism is zero for a very general $t$, then actually it is zero for a general $t$. Also observe that if $j_{t*}=0$ then $j_{\tau(t)*}=0$. So supposing that the bi-Gysin homomorphism is zero or one of $j_{t*}$ or $j_{\tau(t)*}$ is zero we get that from the above
$$A^2(\bcC_{C'}\cup \tau(\bcC)_{C'})\otimes \QQ$$
is weakly representable. That will imply that $A^2(S)\otimes \QQ$ is weakly representable, hence so is $A^2(S)$. That will be a contradiction to our assumption. So if $A^2(S)$ is not weakly representable then for a very general $t$, the kernel of the bi-Gysin homomorphism is countable.

\end{proof}

\subsection{Surfaces with $p_g>0,q=0$ }
The proof of \ref{theorem 1} and \ref{theorem3} applied to the homomorphism $j_{t*}-j_{\tau(t)*}$ will show that the kernel of $j_{t*}-j_{\tau(t)*}$ is either $\{0\}$ or the diagonal $J(C_t)$ for a very general $t$, if $A^2(S)$ is not isomorphic to the albanese variety of $S$. This is the case for surfaces with $p_g>0,q=0$. Note that both the kernels $j_{t*}+j_{\tau(t*)}$ and $j_{t*}-j_{\tau(t*)}$ cannot be the diagonal $J(C_t)$. If it is the case then for any $z$ in $J(C_t)$, we have
$$j_{t*}(z)=\tau(j_{t*})(z)=-j_{t*}(z)$$
which implies that
$$2j_{t*}(z)=0$$
since $q=0$ and Roitman's theorem tells us $j_{t*}(z)=0$, so $j_{t*}$ is the zero map for a very general $t$. This contradicts the fact that the kernel of $j_{t*}$ is countable for a very general $t$, provided that $A^2(S)$ is not isomorphic to the albanese variety of $S$, \cite{BG}[Theorem 19]. Now since $A^2(S)$ is generated by cycles supported on $J(C_t)$, where $C_t$ is a very general smooth hyperplane section of $S$. So either for a very general $C_t$, kernel of both $j_{t*}+j_{\tau(t*)},j_{t*}-j_{\tau(t*)}$ are countable or one of the kernel is $J(C_t)$. So either both the kernels are countable or $\tau_{*}$ acts as $id$ or $-id$ on the image of $J(C_t)$ inside $A^2(S)$. If the later possibility happens then the action of $\tau_*$ on $A^2(S)$, can be detected from its action on the image of the Jacobian of the very general hyperplane section $C_t$. For that we have to exclude the possibility that  the kernels mentioned above, cannot both be countable.


\begin{thebibliography}{AAAAAAA}


\bibitem[Ba]{Ba} K.Banerjee, {\em One dimensional algebraic cycles on non-singular cubic fourfolds in $\PR^5$}, PhD Thesis, University of Liverpool, 2014.

\bibitem[BG]{BG} K. Banerjee and V. Guletskii, {\em Rational equivalence for line configurations on cubic hypersurfaces in $\mathbb P^5$.}, {\small \tt arXiv:1405.6430v1}, 2014.

















\bibitem[BKL]{BKL}S.Bloch, A.Kas, D.Lieberman, {\em Zero cycles on surfaces with $p_g=0$}, Compositio Mathematicae, tome 33, no 2(1976), page 135-145.

\bibitem[CG]{CG} P.Craighero, R.Gattazzo, {\em Quintics surfaces in $\PR^3$ having a non-singular model with $q=p_g=0$, $P_2\neq 0$},Ren. Sem. Math. Uni. Padova, 91:187-198,1994.

\bibitem[DW]{DW} I.Dolgachev, C.Werner, {\em A simply connected numerical Godeaux surface with ample canonical class.},
    {\small \tt arXiv:alg-geom/9704022}

\bibitem[Fu]{Fulton} W. Fulton, {\em Intersection theory}, Ergebnisse der Mathematik und ihrer Grenzgebiete (3),  2.
      Springer-Verlag, Berlin, 1984.

\bibitem[GG]{GG} V.Guletskii, S.Gorchinsky {\em Motives and representability of algebraic cycles on threefolds over a field}, Journal of Algebraic Geometry, 2012, 21, 347-373.

\bibitem[G]{G} V.Guletskii, A.Tikhomirov {\em Algebraic cycles on quadric sections of cubics in $\PR^4$ under the action of symplectomorphisms}, Proc. of the Edinburgh Math. Soc. 59 (2016) 377 - 392.
\bibitem[Ha]{Hartshorne} R. Hartshorne, {\em Algebraic geometry}, Graduate Texts in Mathematics, No. 52. Springer-Verlag,
      New York-Heidelberg, 1977.
\bibitem[HZ]{HZ} F.Hirzebruch, D.Zagier, {\em Classification of Hilbert modular surfaces}, Complex analysis and algebraic geometry, 43-77, Iwanami Shoten, Tokyo, 1977.
\bibitem[HK]{HK} D.Huybrechts, M.Kemeny, {\em Stable maps and Chow groups}, \small { \tt Math.arXiv:1111.1745}

\bibitem[Ko]{Ko} K.Kodaira, {\em On Compact Analytic Surfaces II,III}, Ann. Of Math. (2), 77(1963), 563-626, 78 (1963) 1-40.
\bibitem[M]{M} D.Mumford, {\em Rational equivalence for $0$-cycles on surfaces.}, J.Math Kyoto Univ. 9, 1968, 195-204.

\bibitem[RO]{RO} A.A.Roitman, {\em The torsion of the group of $0$-cycles modulo rational equivalence}, Ann. of Math. (2) 111 (1980), no. 3, 553–569.
\bibitem[SS]M.Schuett, T.Shioda {\em Elliptic surfaces}, \small {\tt math.arXiv:0907.0298}.
\bibitem[Sha]{Sha} I. Shafarevich, et.al,{\em Algebraic surfaces}, Proceedings of the Steklov Mathematics Institute of Mathematics, 75, 1965.
\bibitem[VG]{VG} G.Van Der Geer, {\em Hilbert modular surfaces}, Ergebnisse der Mathematik und ihrer Grenzgeniete, Springer-Verlag, Berlin, 1988
\bibitem[Vo]{Vo} C. Voisin, {\em Complex algebraic geometry and Hodge theory II}, Cambridge University Press, Cambridge, 2003.
\bibitem[Voi]{Voi} C.Voisin, {\em Symplectic invoultions of K$3$ surfaces act trivially on $CH_0$}, Documenta Mathematicae 17, 851-860, 2012.




\end{thebibliography}
\end{document}